\documentclass{amsart}

\usepackage{amsmath,amssymb,amsthm, a4wide}
\usepackage{graphicx,tikz}
\newtheorem{theorem}{Theorem}
\newtheorem*{thm}{Theorem}
\newtheorem*{proposition}{Proposition}
\newtheorem*{corollary}{Corollary}

\newtheorem*{lem}{Lemma}

\theoremstyle{definition}

\theoremstyle{remark}

\DeclareMathOperator{\vol}{vol}
\DeclareMathOperator{\var}{var}
\DeclareMathOperator{\sign}{sign}
\DeclareMathOperator{\diam}{diam}

\begin{document}

\title[]{Wasserstein Distance, Fourier series \\and Applications}

\author[]{Stefan Steinerberger}
\address{Department of Mathematics, University of Washington, Seattle, WA 98195, USA}
\email{steinerb@uw.edu}

\begin{abstract} We study the Wasserstein metric $W_p$, a notion of distance between two probability distributions, from the perspective of Fourier Analysis and discuss applications. In particular,
we bound the Earth Mover Distance $W_1$ between the distribution of quadratic residues in a finite field $\mathbb{F}_p$ and uniform distribution by $\lesssim p^{-1/2}$ (the Polya-Vinogradov inequality implies $\lesssim p^{-1/2} \log{p}$). We also show that for continuous $f:\mathbb{T} \rightarrow \mathbb{R}_{}$ with mean value 0 
$$ (\mbox{number of roots of}~f) \cdot \left( \sum_{k=1}^{\infty}{ \frac{ |\widehat{f}(k)|^2}{k^2}}\right)^{\frac{1}{2}}  \gtrsim \frac{\|f\|^{2}_{L^1(\mathbb{T})}}{\|f\|_{L^{\infty}(\mathbb{T})}}.$$
 Moreover, we show that for a Laplacian eigenfunction $-\Delta_g \phi_{\lambda} = \lambda \phi_{\lambda}$ on a compact Riemannian manifold $W_p\left(\max\left\{\phi_{\lambda}, 0\right\}dx, \max\left\{-\phi_{\lambda}, 0\right\} dx\right)  \lesssim_p \sqrt{\log{\lambda}/\lambda} \|\phi_{\lambda}\|_{L^1}^{1/p}$, which is at most a factor $\sqrt{\log{\lambda}}$ away from sharp. Several other problems are discussed.
\end{abstract}

\maketitle

\vspace{-10pt}

\section{Introduction and Results}

\subsection{The Erd\"os-Turan inequality.} The Erd\"os-Turan inequality \cite{erd1, erd2} is classical: if $\mu$ is a probability measure on $\mathbb{T}$, then for every integer
$n$ we have the following bound 
$$ \sup_{J \subset \mathbb{T} \atop J ~\mbox{\tiny interval}}{ \left| \mu(J) - |J| \right|} \lesssim \frac{1}{n} + \sum_{k=1}^{n}{ \frac{ | \widehat{\mu}(k)|}{k}},$$
where the supremum ranges over all intervals $J \subset \mathbb{T}$ and $\lesssim$ denotes, here and henceforth, the existence of an absolute constant that
does not depend on any other parameters (in particular, here it does not depend on $n$; we refer to Rivat \& Tenenbaum \cite{rivat} for results about the implicit constant). 
Ruzsa \cite{ru1, ru2} has shown that the inequality is best possible.
It would be impossible to summarize all work on the subject, we refer to the higher-dimensional extension of the inequality due to Koksma \cite{koksma} and to the standard
textbooks \cite{dick, drmota, kuipers, mont}.  
 We also mention an inequality
of LeVeque \cite{leveque} which can be derived from the Erd\"os-Turan inequality and is thus generally weaker
$$ \sup_{J \subset \mathbb{T} \atop J ~\mbox{\tiny interval}}{ \left| \mu(J) - |J| \right|} \lesssim \left(\sum_{k=1}^{\infty}{ \frac{ | \widehat{\mu}(k)|^2}{k^2}} \right)^{1/3}.$$

 The goal of this paper is to initiate the study of such estimates for the Wasserstein metric $W_p$. $W_p$ is more refined than the discrepancy notion. The main point of this paper is that these two concepts, Wasserstein
distance and Fourier series, enjoy an interesting interplay. Moreover, while discrepancy has 
been a standard notion in which to measure the regularity of sets, the $p-$Wasserstein distance might provide a natural framework for more refined estimates, that is understanding which types of bounds are available via Fourier Analysis is a natural starting point.

\subsection{Wasserstein distance.} This notion of distance between probability measures was introduced
by Wasserstein in 1969 \cite{wasser} and received its name in a 1970 paper of Dobrushin \cite{dob}. It is now a foundational concept in optimal transport (see e.g. Villani \cite{villani}) as well as probability theory and  partial differential
equations (cf. Otto \cite{otto} and subsequent developments).
We introduce the $p-$Wasserstein distance between two probability measures
$\mu$ and $\nu$ on the one-dimensional torus $\mathbb{T} \cong [0, 2\pi)$ as
$$ W_p(\mu, \nu) = \left( \inf_{\gamma \in \Gamma(\mu, \nu)} \int_{\mathbb{T} \times \mathbb{T}}{ |x-y|^p d \gamma(x,y)}\right)^{1/p},$$
where $| \cdot |$ is the usual distance on the torus and $\Gamma(\mu, \nu)$ denotes the collection of all measures on $\mathbb{T} \times \mathbb{T}$
with marginals $\mu$ and $\nu$, respectively (also called the set of all couplings of $\mu$ and $\nu$). This definition easily carries over to other geometries $M$
in which we case we will considers measures on $M \times M$ (in particular, we will also have $M = \mathbb{T}^d$).
The special case $p=1$ is particularly nice:
by Monge-Kantorovich duality (see e.g. \cite{villani})
$$ W_1(\mu, \nu) = \sup\left\{ \int_{\mathbb{T}}{ f d\mu} - \int_{\mathbb{T}}{ f d\nu}: f~\mbox{is 1-Lipschitz} \right\}.$$
The 1-Wasserstein distance, also sometimes called Earth Mover's Distance, can be interpreted as the total amount of work ($=\mbox{distance}\times \mbox{mass})$ required to
transport $\mu$ to $\nu$. It is not necessarily a quantity that is very easy to compute since it is given by the infimum over a variational problem; in particular, any explicit transport map will result in an upper bound for $W_p$, and various ways of constructing such maps are conceivable (this paper contains two very different methods in \S 6 and \S 9). Correspondingly, the construction of lower bounds is quite a bit harder (and a recurring theme throughout this paper; it would be very nice to have more robust ways of doing this).

\begin{figure}[h!]
\begin{center}
\begin{tikzpicture}[scale=1.4]
\draw [ultra thick] (0,0.7) -- (3,0.7);
\draw [thick] (0,2) to[out=0, in =180] (1,3) to[out=0, in=180] (2,1) to[out=0, in=180] (3,2);
\draw [dashed] (0,2) -- (3,2);
\draw [->] (1, 2.5) -- (2, 1.5);
\node at (4, 2) {$\implies$};
\draw [ultra thick] (5,0.7) -- (8,0.7);
\draw [thick] (5,2) -- (8,2);
\end{tikzpicture}
\end{center}
\caption{Transporting mass to make a distribution uniform.}
\end{figure}
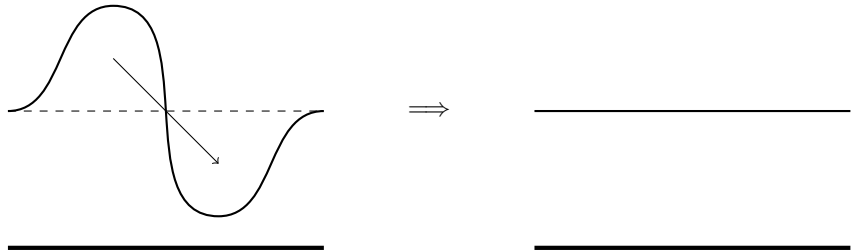

We pose a simple question: let $f:\mathbb{T} \rightarrow \mathbb{R}_{\geq 0}$ be given by a continuous function with mean value $\overline{f}$, how much does it cost to transport the
measure $f(x)dx$ to the constant measure $\overline{f}dx$? How does this depend on the Fourier coefficients of $f$?
The connection to
Fourier analysis seems natural: if $f$ is highly oscillating, then one would expect this transportation distance to be somewhat
smaller since 'being larger than average' and 'being smaller than average' should alternate more frequently.

\subsection{Organization.} The paper is organized as follows. Section \S 2 contains various estimates on the Wasserstein distance in terms of the Fourier coefficients of the function. It also establishes
a connection to the Littlewood conjecture (solved by Konyagin \cite{konyagin} and McGehee-Pigno-Smith \cite{mcgehee}).
Section \S 3 discusses applications: we study the regularity of quadratic residues in finite fields, that of the Kronecker sequence $\left\{n\sqrt{2}\right\}$ and the uncertainty principle mentioned in the abstract. The third part presents a conjecture for Laplacian eigenfunctions, proves part of it up to logarithmic factors and describes a general bound for highly oscillatory functions. $\lesssim$ and $\gtrsim$ always refer to absolute multiplicative constants, whereas $\lesssim_{a,b,c}$ and $\gtrsim_{a,b,c}$ indicate parameters on which the constant may depend. $\sim$ is used
as an abbreviation for $\gtrsim$ as well as $\lesssim$. $\sim_{a,b,c}$ is defined analogously.  Everything in this paper is real-valued, but it will sometimes be convenient to
use complex exponentials; in particular, for all arising functions, $\widehat{f}(k) = \widehat{f}(-k)$.

\section{Wasserstein Distance and Fourier Series}

\subsection{Some Elementary Bounds.} Since the Wasserstein distance is defined by an infimum, it suffices to consider explicit
constructions to obtain upper bounds. If we ignore the condition that this should be done in the most
effective way, we can ask a simpler question: what is the worst possible way it could be done? Since positive
mass has to cancel negative mass, the limiting factor is only how far it could be moved and on the torus $\mathbb{T}$ (or any bounded
domain)
$$  W_p(\max\left\{f(x),0\right\}dx, -\min\left\{f(x),0\right\}dx) \lesssim \diam(\mathbb{T}) \|f\|_{L^1(\mathbb{T})}.$$
However, we could also proceed in a different fashion: we could decompose the function into its Fourier series
$$ f(x) = \sum_{n \geq 1}{a_n \sin{n x} + b_n \cos{n x}},$$
then take the first frequency spanned by $\sin{x}$ and $\cos{x}$ and move the positive bump to the negative bump. We observe that this could
lead to the creation of a new signed measure if the function $f$ happens to be negative in the region where $a_1 \sin{x} + b_1 \cos{x}$ is positive,
however, this increases subsequent transport cost and is thus admissible in proving upper bounds.
This process
creates a new function that is hopefully somewhat closer to its mean value
$$ f_2(x) =  \sum_{n \geq 2}{a_n \sin{n x} + b_n \cos{n x}}$$
and, clearly, the transportation cost is $\lesssim \pi \sqrt{a_1^2 + b_1^2}^{}.$
We will now iterate the procedure and observe that it gets cheaper: once we are dealing with frequency $k$, the transport cost per unit of mass decreases to $\sim k^{-1}$
since the function oscillates at a higher frequency. 
\begin{center}
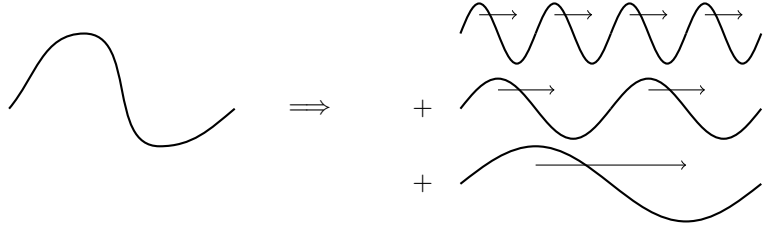
\begin{figure}[h!]
\begin{tikzpicture}
\draw [thick] (-1,1) to[out=50, in =180] (0,2) to[out=0, in=180] (1,1-0.5) to[out=0, in = 220] (2,1);
\node at (3,1) {$\implies$};
\node at (4.5,1) {$+$};
\node at (4.5,0) {$+$};
    \draw[thick] (5,0) sin (6,0.5) cos (7,0) sin (8,-0.5) cos (9,0);
    \draw[thick] (5,1) sin (5.5,1+0.4) cos (6,1+0) sin (6.5,1-0.4) cos (7,1+0);
    \draw[thick] (7,1) sin (7.5,1+0.4) cos (8,1+0) sin (8.5,1-0.4) cos (9,1+0);
    \draw[thick] (5,2) sin (5.25,2+0.4) cos (5.5,2+0) sin (5.75,2-0.4) cos (6,2+0);
    \draw[thick] (6,2) sin (6.25,2+0.4) cos (6.5,2+0) sin (6.75,2-0.4) cos (7,2+0);
    \draw[thick] (7,2) sin (7.25,2+0.4) cos (7.5,2+0) sin (7.75,2-0.4) cos (8,2+0);
    \draw[thick] (8,2) sin (8.25,2+0.4) cos (8.5,2+0) sin (8.75,2-0.4) cos (9,2+0);
\draw [->] (5.25, 2.25) -- (5.75, 2.25);
\draw [->] (6.25, 2.25) -- (6.75, 2.25);
\draw [->] (7.25, 2.25) -- (7.75, 2.25);
\draw [->] (8.25, 2.25) -- (8.75, 2.25);
\draw [->] (5.5, 1.25) -- (6.25, 1.25);
\draw [->] (7.5, 1.25) -- (8.25, 1.25);
\draw [->] (6, 0.25) -- (8, 0.25);
\end{tikzpicture}
\caption{Explicit frequency-by-frequency transport instructions.} 
\end{figure}
\end{center}
\vspace{-10pt}
Iterating this procedure and using the triangle inequality suggests
$$  W_1(f(x)dx, \overline{f} dx) \lesssim \sum_{k=1}^{\infty}{ \frac{|\widehat{f}(k)|}{k}}$$
which is eerily similar to the Erd\H{o}s-Turan inequality. This is no coincidence.
\begin{proposition} For all $n \in \mathbb{N}$ and all $f:\mathbb{T} \rightarrow \mathbb{R}_{\geq 0}$ with mean value $\overline{f}$,
$$  W_1(f(x)dx, \overline{f} dx) \lesssim \frac{\|f\|_{L^1}}{n} +  \sum_{k=1}^{n}{ \frac{|\widehat{f}(k)|}{k}}$$
\end{proposition}
\begin{proof} The proof is so short that we can discuss it right here. We assume w.l.o.g. $\|f\|_{L^1} = 1$ and thus $\overline{f} = (2\pi)^{-1}$.
 We can approximate $\mu = f(x)dx$ by a measure $\nu$ given as the sum of a large number of normalized Dirac point masses (as long as we
do not require a bound on the number). Setting
$$ \nu = \frac{2\pi}{N} \sum_{k=1}^{N}{ \delta_{x_k}}$$
and recalling the Koksma-Hlawka inequality (see e.g. \cite{kuipers})
$$\left| \frac{1}{N} \sum_{k=1}^{N}{g(x_k)} - \int_{\mathbb{T}} g(x)  dx \right| \leq \var(g)  \sup_{J \subset \mathbb{T} \atop J ~\mbox{\tiny is interval}}{ \left| \mu(J) - |J| \right|},$$
where $\var$ denotes the total variation of the function. Using $\var(g) \leq 2\pi$ for $1-$Lipschitz function on the Torus together with the description of $W_1$ via Monge-Kantorovich duality
$$ W_1(\mu, dx) = \sup\left\{ \int_{\mathbb{T}}{ g d\mu} - \int_{\mathbb{T}}{ g dx}: g~\mbox{is 1-Lipschitz} \right\} \leq 2\pi   \sup_{J \subset \mathbb{T} \atop J ~\mbox{\tiny is interval}}{ \left| \mu(J) - |J| \right|}$$
shows that the $W_1$ distance can be bounded from above by the discrepancy. The result then follows from an application of the Erd\H{o}s-Turan inequality.
\end{proof}

A priori one would perhaps be inclined to believe that the procedure sketched
above to motivate the Erd\H{o}s-Turan analogue for Wasserstein distances should be extremely lossy: frequencies are essentially treated in complete isolation from each
other, the triangle inequality is frequently employed and many redundancies are accumulated (and, indeed, all of this is true and the inequality \textit{is} lossy). However, the estimate is, when applied to trigonometric polynomials, at most a constant factor larger than the 
trivial bound $\|f\|_{L^1}$. A general smoothing procedure, introduced in \S 4.2, allows to reduce general measures to the case
of trigonometric polynomials.

\begin{thm}[McGehee-Pigno-Smith solution of the Littlewood conjecture \cite{mcgehee}] We have, for any finite set of integers $\lambda_1 < \lambda_2 < \dots < \lambda_N$ and any numbers $a_1, \dots, a_N$
$$\left\| \sum_{k=1}^{N}{ a_k e^{i \lambda_k t}} \right\|_{L^1(\mathbb{T})}  \gtrsim \sum_{k=1}^{N}{ \frac{ |a_k|}{k}} $$
\end{thm}

Littlewood's conjecture is concerned with the intuition that the $L^1-$norm of a function with Fourier coefficients that are not very small should not be too small either: the
Dirichlet kernel should be an extremal example. More precisely, Littlewood asked \cite{hardy} whether it was true that for distinct integers $\lambda_1 < \dots < \lambda_N$ and coefficients $|a_i| \geq 1$
$$  \left\| \sum_{k=1}^{N}{ a_k e^{i \lambda_k t}} \right\|_{L^1(\mathbb{T})} \gtrsim \log{N}.$$
This was independently proven by Konyagin \cite{konyagin} and McGehee-Pigno-Smith \cite{mcgehee}. A nice presentation of the McGehee-Pigno-Smith argument can be found in the charming book of Choimet \& Queffelec \cite{cho}.
We can reinterpret the McGehee-Pigno-Smith inequality as stating that the trivial frequency-by-frequency transport of trigonometric polynomials
to their mean value is never less efficient (in the $W_1$ distance) than the trivial $L^1-$bound.

\subsection{The 2-Wasserstein distance.} As is perhaps not surprising, $W_2$ is special:  the linearized $W_2$ distance and Sobolev space $\dot H^{-1}$ are naturally related (see Otto \& Villani \cite[\S 7]{otto} or the textbook of Villani \cite[\S 7.6]{villani}): if $\mu$ is a measure and $d\mu$ is a small perturbation, then
$$ W_2(\mu, \mu + d\mu) = \| d\mu\|_{\dot H^{-1}(\mu)} + o(\|d\mu\|),$$
where the $o(\|d \mu\|)$ term goes faster to 0 than linearly as $\|d\mu\| \rightarrow 0$.
A rigorous formulation can be found in the book of Villani \cite[Exercise 22.20]{villani2} and states that for $h$ being smooth, bounded and compactly supported
$$ \|h\|_{H^{-1}} = \lim_{\varepsilon \rightarrow 0}{ \frac{W_2((1+\varepsilon h) \nu, \nu)}{\varepsilon}}.$$
This follows from the more general Benamou-Brenier formula \cite{ben} which states that for two positive measure $\mu, \nu$
$$ W_2(\mu, \nu) = \inf\left\{ \int_0^1 \| d \mu_t \|_{\dot H^{-1}(\mu_t)}: \mu_0 = \mu, \mu_1 = \nu\right\},$$
where the infimum ranges over all one-parameter flows from $\mu$ to $\nu$.
R. Peyre \cite{peyre} used this to show that for any two probability measure $\mu, \nu$
$$ W_2( \mu, \nu) \leq 2 \| \mu - \nu \|_{\dot H^{-1}(\mu)}.$$
Peyre's paper \cite{peyre}, which also establishes corresponding lower bounds under some additional assumptions, is discussed in the book of Santambrogio \cite{santa}.
 We give a quick summary of the argument in \S 4.2. and discuss a modification that allows for an additional cut-off of Erd\H{o}s-Turan type.
Santambrogio describes in \cite[Ex. 64]{santa} a more precise result for probability distributions $g(x)dx$ on $[0,1]$
$$ W_2(g(x)dx, dx) = \| 1 - g\|_{\dot H^{-1}([0,1])}.$$
We remark that \cite[Ex. 64]{santa} also holds for discrete measures: any such discrete measure can be approximated by a sequence of mollifications and all arising quantities are stable under taking that limit (this is part of a recurring theme that will also be further explored in \S 4.2).

\subsection{General $p-$Wasserstein bounds.} This section describes a bound for $W_p$ in the one-dimensional setting. This is accomplished via a more complicated transport plan
that makes better use of the underlying cancellation occuring between frequencies. 

\begin{theorem} Let $1 \leq p < \infty$, let $\mu$ be an absolutely continuous measure and let $d\mu$ be a perturbation, then
$$   W_p(\mu, \mu +  d\mu) \lesssim_{p}    \left( \sum_{k=1}^{\infty}{ \frac{|\widehat{d\mu}(k)|^2}{k^{2p-2} }}\right)^{\frac{1}{2p}}.$$
\end{theorem}
The proof of Theorem 1 can be found in \S 5. The main interest of Theorem 1 is presumably the transport plan itself that is fairly nontrivial and might have applications in other settings. Note added in print: Cole Graham has
informed us that he has since obtained a result that should be compared with Theorem 1
$$ W_p(\mu, \nu) \leq c_p \|\mu - \nu\|_{\dot W^{-1,p}}$$
as well as an associated Erd\H{o}s-Turan inequality for all $p \geq 2,$
$$ \|\mu - \nu\|_{\dot W^{-1,p}} \leq \frac{c_p}{n} + c_p \left(\sum_{k=1}^{n-1}{ \frac{|\widehat \mu(k)|^q}{k^q}} \right)^{\frac{1}{q}},$$
 where $q$ is the H\"older dual of $p$. It would be desirable to have more results of this flavor.
An entirely different problem, that becomes especially relevant when trying to understand the quality of results of this type, is the existence of lower bounds. Explicit transport plans can be used for the construction of upper bounds. We now discuss lower bounds (as customary, obtaining matching lower bounds is likely the more difficult problem).
\begin{theorem} Let $f: \mathbb{T} \rightarrow \mathbb{R}_{\geq 0}$ with mean value $\overline{f}$. Then
$$ W_1(f(x)dx, \overline{f} dx) \gtrsim  \frac{1}{\|f\|_{L^{\infty}(\mathbb{T})}}\sum_{k \in \mathbb{Z} \atop k \neq 0}{ \frac{1 + \log{|k|}}{k^2} |\widehat{f}(k)|^2 }.$$
\end{theorem}
The proof of this result is contained in \S 6.
By H\"older's inequality, this bound
for $W_1$ implies a lower bound for $W_p$ for all $p \geq 1$ but one would, of course, assume that in most circumstances the
Wasserstein distance is actually growing with $p$. It would be nice if one had lower bounds reflecting this sort of behavior.

\begin{quote}
\textbf{Problem.} How does one effectively prove \textit{lower} bounds on Wasserstein distances?
\end{quote}
A natural approach is duality, see Loeper \cite{loeper}, Maury \& Venel \cite{maury} or the discussion in the book of Santambrogio \cite[\S 5.5.2]{santa}. This, however, seems to shift the problem towards the question of how one would go about finding good test functions which does not seem easier.

\section{Applications}

\subsection{ Number Theory.}

These results have nontrivial implications in Analytic Number Theory. Despite being a standard notion in probability theory and statistics, the Wasserstein distance does not
seem to be frequently used in Number Theory (except for the $W_1-$distance between a set of numbers and an ordered set of numbers which appears frequently, see for example \cite[Theorem 2.1.4]{kuipers}); two
recent papers where it appears in a more substantial role are \cite{mehr, saks}. 
\begin{center}
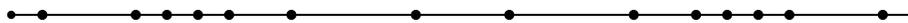
\begin{figure}[h!]
\begin{tikzpicture}[scale=1.2]
\draw [thick] (0,0) -- (10,0);
\filldraw (0,0) circle (0.04cm);
\filldraw (10/29,0) circle (0.05cm);
\filldraw (40/29,0) circle (0.05cm);
\filldraw (50/29,0) circle (0.05cm);
\filldraw (60/29,0) circle (0.05cm);
\filldraw (70/29,0) circle (0.05cm);
\filldraw (90/29,0) circle (0.05cm);
\filldraw (130/29,0) circle (0.05cm);
\filldraw (160/29,0) circle (0.05cm);
\filldraw (200/29,0) circle (0.05cm);
\filldraw (220/29,0) circle (0.05cm);
\filldraw (230/29,0) circle (0.05cm);
\filldraw (240/29,0) circle (0.05cm);
\filldraw (250/29,0) circle (0.05cm);
\filldraw (280/29,0) circle (0.05cm);
\end{tikzpicture}
\caption{The renormalized quadratic residues in $\mathbb{F}_{29}$ displayed on $[0,1]$. Every dot except the one in zero represents two quadratic residues. How costly is it to move this point measure to the uniform distribution?}
\end{figure}
\end{center}
We illustrate the usefulness of this notion with an example on $\mathbb{T} \cong [0,1]$: let $p$ be an odd prime, what can be said  about the distribution properties of the set of quadratic residues in the finite field $\mathbb{F}_p$ rescaled to the unit interval $[0,1]$? This set can be explicitly written as $\left\{ \left\{ n^2/p\right\}: 1 \leq n \leq p \right\} \subset [0,1]$, where
$\left\{ x \right\} = x - \left\lfloor x \right\rfloor$ denotes the fractional part. 

 It is classical (see e.g. \cite{mont}) that this measure is rather evenly distributed and that the discrepancy of their distribution can be analyzed using the Erd\H{o}s-Turan inequality: we write 
$$ \mu = \frac{1}{p} \sum_{k=1}^{p}{ \delta_{  \left\{ k^2/p\right\}}}.$$
The $j-$th Fourier coefficient can be written in explicit form
$$ \widehat{\mu}(j) = \int_{0}^{1}{ e^{-2 \pi i j x} d\mu} = \frac{1}{p}\sum_{k=1}^{p}{ e^{-2 \pi i j \frac{k^2}{p}}}.$$
This is a Gauss sum for which a closed form expression is available and
$$ |\widehat{\mu}(j)| = \begin{cases} 1 \qquad &\mbox{if}~p~\mbox{divides}~j \\ p^{-1/2} \qquad &\mbox{otherwise.} \end{cases}$$
We use this identity in combination with the Erd\H{o}s-Turan inequality 
\begin{align*}
 \sup_{J \subset \mathbb{T} \atop J ~\mbox{\tiny is interval}}{ \left| \mu(J) - |J| \right|} &\lesssim \frac{1}{n} + \sum_{k=1}^{n}{ \frac{ | \widehat{\mu}(k)|}{k}}\\
&\leq  \frac{1}{n} + \sum_{1 \leq k \leq n \atop p|k}{\frac{1}{k}} + \frac{1}{\sqrt{p}}\sum_{1 \leq k \leq n}{\frac{1}{k}}\\
&\lesssim  \frac{1}{n} +  \frac{1}{p} \log{\left( \frac{n}{p} \right)}  + \frac{\log{n}}{\sqrt{p}}
\end{align*}
Optimizing in $n$ leads to $n=p$ and 
$$  \sup_{J \subset \mathbb{T} \atop J ~\mbox{\tiny is interval}}{ \left| \mu(J) - |J| \right|} \lesssim \frac{\log{p}}{\sqrt{p}}.$$
This can be seen as a special case of the Polya-Vinogradov estimate for character sums. However, we can obtain a refinement for the $W_2-$distance by using Santambrogio's exercise
$$ W_2(g(x)dx, dx) = \| 1 - g\|_{\dot H^{-1}([0,1])}$$
to conclude that
\begin{align*}
W_2(\mu, dx) \leq  \left(\sum_{k \in \mathbb{Z} \atop k \neq 0}{ \frac{|\widehat{\mu}(k)|^2}{|k|^2}}\right)^{1/2} \leq \left(\sum_{k \in p\mathbb{Z} \atop k \neq 0}{ \frac{1}{|k|^2}} + \sum_{k \in \mathbb{Z}\setminus p\mathbb{Z}}{\frac{1}{p} \ \frac{1}{|k|^2}}  \right)^{1/2} \lesssim \frac{1}{\sqrt{p}}.
\end{align*}

Put differently, while irregularities in the distribution of quadratic residues may occur, they do not occur frequently. This approach might be applicable for many other examples of naturally occuring objects in Analytic Number Theory.
\begin{quote}
\textbf{Open Problem.} Establish Wasserstein bounds from above and below for the regularity of sequences and sets appearing in Number Theory.
\end{quote}
Since the Wasserstein distance is smaller than notions of discrepancy, another natural problem is to ask to which extent irregularity of distribution phenomena continue to hold. We recall that if $(x_n)_{n = 1}^{\infty}$
is a sequence on $[0,1]$, then the sequence of measures
$$ \mu_N = \frac{1}{N} \sum_{n=1}^{N}{\delta_{x_n}}$$
satisfies
$$\sup_{J \subset \mathbb{T} \atop J ~\mbox{\tiny is interval}}{ \left| \mu_{N}(J) - |J| \right|}  \gtrsim \frac{\log{N}}{N} \quad 
\mbox{for infinitely many}~N.$$
The sequence of measures so obtained cannot be uniformly regular -- this phenomenon is sometimes refered to as \textit{Irregularities of Distributions}. This particularly celebrated result is due to W. M. Schmidt \cite{schm}. The factor $\log{N}$ is sharp, many examples attaining it are known (see \cite{dick, drmota, kuipers}). The problem is wide open in higher dimensions \cite{bil1, bil2, bil3, bil4, bil5, bil6}. 
An example of a sequence with the highest possible degree of regularity is the Kronecker (or $n\alpha-$) sequence  $\left(\left\{n \alpha\right\}\right)_{n=1}^{\infty}$. Here, $\alpha \in \mathbb{R}$ is assumed to have the property that there exists $c>0$ $$ \forall \frac{p}{q} \in \mathbb{Q} \qquad \left| \alpha - \frac{p}{q} \right| \geq \frac{c}{q^2}.$$
A classical example of such a \textit{badly approximable} number is $\sqrt{2}$. 
\begin{theorem}[Earth Mover Distance for Kronecker sequences] Let $\alpha$ be badly approximable, then
$$ W_2\left(  \frac{1}{N} \sum_{n=1}^{N}{\delta_{\left\{ n \alpha \right\}}} , dx\right) \lesssim_{\alpha} \frac{(\log{N})^{\frac{1}{2}}}{N}.$$
\end{theorem}
This is a factor $(\log{N})^{1/2}$ smaller than what is possible in the discrepancy regime. It implies that Kronecker sequences, while being unable to avoid discrepancy $\sim (\log{N})/N$) can manage to avoid having it on sets of large measure. We quickly recall the notion of the discrepancy function for a set of points $\left\{x_1, \dots, x_N\right\} \subset [0,1]$ given by
$$ d(x) = \left| \frac{ \# \left\{1 \leq i \leq N: 0 \leq x_i \leq x\right\}}{N} - x\right|.$$
The star-discrepancy $D_N$ of the set is then defined as $D_N = \max_{0 \leq x \leq 1}{d(x)}$ and the $L^p-$discrepancy $D_N^{(p)}$ is given by
$$ D_N^{(p)} = \left( \int_{0}^{1}{d(x)^p dx} \right)^{\frac{1}{p}}.$$
The $L^p-$discrepancy of the Kronecker sequence is $\mathcal{O}_p(\sqrt{\log{N}}/N)$ for $1 \leq p < \infty$ but is at least a factor $\sqrt{\log{N}}$ larger for $p=\infty$. Wasserstein distances should be related to notions of $L^p-$discrepancy (at least in one dimension) which would naturally connect these results; we believe this to be a promising question that should be within reach. It also seems likely that the connection between star-discrepancy and $W^p-$Wasserstein distance is not quite as pronounced in higher dimensions though this remains to be understood. Can Theorem 3 be further improved? Can it be extended to higher dimensions? \\

We conclude this section by emphasizing another question that could be quite interesting: it is known that for a universal $c>0$, we have $D_N(x_n) \geq c \log{(N)} N^{-1}$ for infinitely many $N$. 
 Is there an Irregularities of Distribution phenomenon for the Wasserstein distance?
\begin{quote}
\textbf{Open problem.} Let $(x_n)_{n=1}^{\infty}$ be a sequence of points on $[0,1]$ and define $\mu_N = N^{-1} \sum_{k=1}^{N}{\delta_{x_k}}$. Is it possible for the sequence $  N \cdot W_p(\mu_N, dx)$ to stay bounded?
\end{quote}
The case $p=1$ is special since unboundedness there would imply unboundedness for all $p>1$. We do not know how hard this question is. One would perhaps be inclined to conjecture that $ W_p(\mu_N, dx) \gtrsim c_p \sqrt{\log{N}}/N$ for infinitely many $N$ for some $c_p > 0$ depending only on $p$. It would be of substantial interest to have results of this flavor in higher dimensions as well. 

\subsection{An Uncertainty Principle.} This section discusses a result that is quite unrelated to the notion of Wasserstein distance and was, somewhat to our surprise, obtained as a byproduct.
A classical result often mentioned in the context of the Sturm Oscillation Theorem is that a continuous function $f$ only comprised of large frequencies has many roots: more precisely, if
$$ f(x) = \sum_{k=n}^{\infty}{a_{k} \sin{(k x)} + b_k \cos{(k x)}},~\mbox{then} \qquad \# \left\{x \in \mathbb{T}: f(x) = 0 \right\} \gtrsim n.$$
A proof via complex analysis is immediate: complexify, factor out the largest monomial and use the argument principle. There is also a purely 'real' proof via partial differential equations (see \cite{stein}).
We now discuss a more quantitative variant capturing the same phenomenon. 

\begin{theorem} For all continuous $f:\mathbb{T} \rightarrow \mathbb{R}_{}$  with mean value 0
$$ (\emph{number of roots of}~f) \cdot \left( \sum_{k=1}^{\infty}{ \frac{ |\widehat{f}(k)|^2}{k^2}}\right)^{\frac{1}{2}}  \gtrsim \frac{\|f\|^{2}_{L^1(\mathbb{T})}}{\|f\|_{L^{\infty}(\mathbb{T})}}.$$
\end{theorem}
Setting $f(x) =  \sin{(n x)}$ shows that the inequality is sharp up to constants.
We believe it to be quite curious -- it can be interpreted as a quantitative Sturm oscillation estimate and we believe that it raises a series of natural questions: to which extent can similar results be proven in a more general setting? Are the $L^p-$norms on the right-hand side optimal? Is there a more elementary proof not passing through the notion of Wasserstein distance? We observe that the result improves on the classical
result for functions of the type, say,
$$ f(x) = \varepsilon \sin{(nx)} + \sin{(n^2 x)}.$$
The classical Sturm Oscillation Theorem can only deduce the existence of at least $n$ roots while Theorem 5 implies the existence of $\gtrsim n^2/(1+\varepsilon^2 n^2)^{1/2}$ roots which is better than the classical bound as soon as $\varepsilon \lesssim n^{-1}$.
We observe that there is a natural corollary for differentiable functions that bypasses the notion of both Fourier series or Sobolev spaces with negative indices. 
\begin{corollary} Let $f:\mathbb{T} \rightarrow \mathbb{R}$ be continuously differentiable with mean value $\overline{f}$. Then
$$ (\emph{number of critical points of}~f) \cdot \|f- \overline{f}\|_{L^2(\mathbb{T})}  \gtrsim \frac{\|f'\|^{2}_{L^1(\mathbb{T})}}{\|f'\|_{L^{\infty}(\mathbb{T})}}.$$
\end{corollary}
The corollary follows immediately by applying Theorem 4 to $f'$ and  $f(x) =  \sin{(n x)}$ again shows the inequality to be sharp. It seems again to capture a quite fundamental
idea: if the absolute value of the derivative happens to be large on a set of large measure, then either the function keeps going up (making the $L^2-$norm big) or the derivative
is alternating quite a bit between positive and negative values (creating critical points). It would again be desirable to have a simple proof avoiding the Wasserstein distance altogether
and it would be fascinating to understand whether the inequality (or rather, an inequality encapsulating the same principle) exists in higher dimensions.
However, even the one-dimensional case is nontrivial.
\begin{quote} \textbf{Question.} For which triples $1 \leq p,q,r \leq \infty$ is there an estimate
$$ (\mbox{number of roots of}~f) \cdot \left( \sum_{k=1}^{\infty}{ \frac{ |\widehat{f}(k)|^2}{k^2}}\right)^{\frac{1}{2}}  \gtrsim \frac{\|f\|^{}_{L^p(\mathbb{T})} \|f\|^{}_{L^q(\mathbb{T})} }{\|f\|_{L^{r}(\mathbb{T})}}?$$
\end{quote}
The proof of Theorem 4 actually shows a slightly stronger result
$$ \# \left\{x \in \mathbb{T}: f(x) = \overline{f} \right\} W_1(f(x)dx, \overline{f} dx) \gtrsim \frac{\|f-\overline{f}\|_{L^1}^2}{\|f\|_{L^{\infty}}},$$
where $\overline{f}$ is the average of $f$.
The author recently established \cite{steini} a variant of Theorem 4 in two dimensions:  if
$f(x)dx$ and $g(x)dx$ are two absolutely continuous measures on a two-dimensional domain $M$ with continuous densities and the same total mass, then, for all $1 \leq p <\infty$,
$$ W_p(f(x)dx, g(x) dx) \cdot \mathcal{H}^1 \left\{x \in M: f(x) = g(x) \right\} \gtrsim_{M,p} \frac{\|f-g\|_{L^1(M)}^{1+1/p}}{\|f-g\|_{L^{\infty}(M)}},$$
where $\mathcal{H}^1$ is the one-dimensional Hausdorff measure. This turns out to be an essential ingredient in establishing a metric analogue of classical Sturm-Liouville theory in two dimensions. 

\subsection{Laplacian eigenfunctions and Applications.} Throughout this section, let $(M,g)$ be a smooth, compact Riemannian manifold without boundary and let $(\phi_k)_{k=0}^{\infty}$
denote the sequence of $L^2-$normalized Laplacian eigenfunctions. These objects are of fundamental importance in mathematical physics because they
diagonalize the Laplacian. Ever since the seminal 1912 paper of Weyl \cite{weyl}, they are known to interact deeply with the local and global geometry of the manifold
(see \cite{xiu} for a recent fascinating mystery). We refer to a recent book of Zelditch \cite{zeld2} for an overview.
This section discusses a new property of these functions. Any solution of $-\Delta_g \phi = \lambda \phi$ is known to oscillate on scale $\sim \lambda^{-1/2}$
(also known as the wavelength). Can the positive part $\phi_{+} = \max\left\{\phi,0\right\}$ be moved to the negative part $\phi_{-} = \max\left\{-\phi,0\right\}$ by essentially
moving the entire $L^1-$mass by roughly one wavelength?

\begin{quote}
\textbf{Conjecture.} Is it true that if $-\Delta \phi = \lambda \phi$ on $(M,g)$, then
$$ W_p\left(\max\left\{\phi_{\lambda}, 0\right\}dx, \max\left\{-\phi_{\lambda}, 0\right\} dx\right) \sim_{p, (M,g)} \frac{1}{\sqrt{\lambda}} \|\phi\|^{\frac{1}{p}}_{L^1(M)}?$$
\end{quote}
We remark that the conjecture is two-sided and claims both upper and lower bounds.
This question seems already nontrivial on $\mathbb{T}^2$: it may be easily verified for $\cos{ \left( \left\langle k, x\right\rangle\right)}$ (where $k \in \mathbb{Z}^2$) but the eigenvalues of $-\Delta_{\mathbb{T}^2}$ can
have arbitrarily large multiplicity and it seems nontrivial to establish the result for linear combinations. The same phenomenon occurs on $\mathbb{S}^2$. It seems conceivable
that these cases already encapsulate all typical obstructions and that a proof on either $\mathbb{S}^2$ or $\mathbb{T}^2$ could illustrate the general case.
 We prove a result that makes the upper bound seem likely.

\begin{theorem} \label{eig} 
Let $(M,g)$ be a compact Riemannian manifold and $\partial M = \emptyset$. 
 If $1 \leq p < \infty$ and $-\Delta_g \phi = \lambda \phi$ on $M$, then
$$ W_p\left(\max\left\{\phi_{\lambda}, 0\right\}dx, \max\left\{-\phi_{\lambda}, 0\right\} dx\right) \lesssim_{p,(M,g)} \sqrt{\frac{ \log{\lambda}}{\lambda}} \|\phi\|^{\frac{1}{p}}_{L^1(M)}.$$
Moreover,  if $f:M \rightarrow \mathbb{R}$ can be written as
$$ f = \sum_{\lambda_k \geq \lambda}{ \left\langle f, \phi_k \right\rangle \phi_k},$$
then
$$  W_p\left(\max\left\{f(x), 0\right\}dx, \max\left\{-f(x), 0\right\} dx\right)  \lesssim_{p, (M,g)}  \frac{1}{\sqrt{\lambda}}   \sqrt{   \log\left(  \lambda \frac{\|f\|_{L^2(M)}}{\|f\|_{L^1(M)}}   \right)} \|f\|^{\frac{1}{p}}_{L^1}.$$
\end{theorem}
This establishes one direction of the conjecture up to a logarithmic factor. Unfortunately, the argument does not seem to be able to
produce sharper results and different techniques might be needed. The second part of the statement, adapting a technique from \cite{steinos}, shows that all 'high-frequency' functions
have positive and negative mass cancelling out with little work, this seems like a fairly fundamental principle; it is not clear to us whether
the logarithmic factor is necessary.\\

The conjectured lower bound in Theorem 5 may be within reach in the special case of two-dimensional manifolds for certain values of $p$. The relevant question that needs to be answered seems to be the following: suppose $-\Delta \phi = \lambda \phi$ on a two-dimensional manifold $M$ and let $A = \left\{x \in M: \phi(x) = 0\right\}$ be the set where the eigenfunction vanishes, is then necessarily a fixed amount of the $L^p-$mass far away from the roots? More precisely, do we have an estimate of the type
$$ \| \phi \|_{L^p\left(\left\{x \in M: d(x, A) \geq \delta\cdot \lambda^{-1/2}\right\}\right)} \gtrsim \| \phi\|_{L^p(M)}?$$
An even stronger statement would be the following: if $\Omega \subset \mathbb{R}^2$ and $\phi$ is the first non-trivial eigenfunction of the Dirichlet-Laplacian on $\Omega$, i.e. $-\Delta \phi = \lambda_1 \phi$, is it true that
$$ \| \phi \|_{L^p\left(\left\{x \in M: d(x, A) \geq \delta\cdot \lambda_1^{-1/2}\right\}\right)} \gtrsim \| \phi\|_{L^p(M)}?$$
This estimate can only hold in two dimension due to a well known two-dimensional rigidity phenomenon. Existing results imply that the estimate holds for $p=\infty$ (see \cite{georg, lieb, lierl,manas}).

\section{Peyre's argument and a smoothing procedure}
\subsection{Peyre's argument.}
 We quickly describe Peyre's argument; while presentation has been slightly changed, all the content is from \cite{peyre}.
The result can be interpreted as a non-asymptotic connection between the linearized $W_2$ distance and $\dot H^{-1}$ (see, for example,  Otto \& Villani \cite{otto} or the textbook of
Villani \cite[\S 7.6]{villani}): if $\mu$ is a measure and $d\mu$ is an infinitesimally small perturbation, then
$$ W_2(\mu, \mu + d\mu) = \| d\mu\|_{\dot H^{-1}(\mu)} + o(\left\| d\mu \right\|).$$
We recall the definition of the weighted $\dot H^1-$norm 
$$ \|g \|_{\dot H^1(\mu)} := \left(\int_{\mathbb{T}}{ |\nabla g|^2 d\mu}\right)^{\frac{1}{2}}$$
which then defines $\dot H^{-1}(\mu)$ via duality
$$ \| h\|_{\dot H^{-1}(\mu)} := \sup \left\{ \left\langle g, h\right\rangle:  \| g\|_{\dot H^1(\mu)} \leq 1\right\}.$$
If $\mu_t$ is now a one-parameter family interpolating between $\mu_0 = \mu$ and $\mu_1 = \nu$, then by the Benamou-Brenier formula (see Section 2.2),
$$ W_2(\mu, \nu) \leq \int_{0}^{1}{\left\| \frac{\partial\mu_t}{\partial t}\right\|_{\dot H^{-1}(\mu_t)}} dt$$
and the Benamou-Brenier formula \cite{ben} (also discussed in the textbook of Villani \cite[\S 8.1]{villani}) implies that taking the infimum over all possible curves actually coincides with the $W_2$ distance (we will not use this fact).
We use the bound for the choice of curve $\mu_t = (1-t) f(x)dx + t \overline{f} dx$, where
 $f$ is the density of the absolutely continuous probability measure $\mu$ and $\overline{f}$ denotes the average value of $f$ 
$$ W_2(f(x)dx, \overline{f} dx) \leq \int_{0}^{1}{ \|  f(x) -  \overline{f}\|_{\dot H^{-1}(\mu_t)} dt}.$$
We observe that $\mu_t \geq t \overline{f} dx$ and thus, for any $g$,
$$ \| g\|_{\dot H^1(\mu_t)} := \left(\int_{\mathbb{T}}{ |\nabla g|^2 d\mu}\right)^{\frac{1}{2}} \geq  \left(\overline{f} t \int_{\mathbb{T}}{ |\nabla g|^2 dx}\right)^{\frac{1}{2}} = \sqrt{\overline{f} t} \|g\|_{\dot H^1(dx)}$$
which, by duality, then implies
$$ \| f  - \overline{f} \|_{\dot H^{-1}(\mu_t)} \leq \frac{1}{\sqrt{\overline{f} t} } \| f  - \overline{f} \|_{\dot H^{-1}(dx)} =  \frac{1}{\sqrt{\overline{f} t} } \| f  - \overline{f} \|_{\dot H^{-1}(dx)} .$$
This implies
\begin{align*}
 W_2(f(x)dx, \overline{f} dx)& \leq \int_{0}^{1}{ \| f - \overline{f}\|_{\dot H^{-1}(\mu_t)} dt} \leq    \|  f \|_{\dot H^{-1}(dx)} \int_{0}^{1}{ \frac{ dt}{\sqrt{\overline{f} t }}} = 2\frac{\|  f \|_{\dot H^{-1}}}{ \sqrt{\overline{f}}} =  2\frac{\|  f \|_{\dot H^{-1}(dx)}}{ \|f\|_{L^1}^{1/2}}.
\end{align*}

\subsection{A Smoothing Procedure.} The purpose of this section is to explain how a basic smoothing procedure implies that it is always possible to introduce a free parameter as in the Erd\H{o}s-Turan type inequality. It also shows that in most cases it will be sufficient to consider trigonometric polynomials (as long as the argument does not depend on their degree).
This approach is universal and not limited to any particular Theorem; indeed, it will be implicitly assumed in most subsequent arguments.
Let $f:\mathbb{T}^d \rightarrow \mathbb{R}_{\geq 0}$ be given and assume its Fourier series is
$$ f(x) = \sum_{k \in \mathbb{Z}^d}{\widehat{f}(k) e^{2\pi i \left\langle k, x \right\rangle}}.$$
As we do throughout the paper, we use $\overline{f} = \widehat{f}(0)$ to denote the mean value of the function.
The main idea is to apply the heat equation to smooth the function and then interpret the outcome $g$ of this smoothing process as the result of applying a particular transport plan moving $f$ to $g$ (we refer to the beginning of \S 9 where this argument is described in greater detail). If we can bound the $p-$Wasserstein cost of moving $f$ to $g$, we may use the triangle inequality to estimate
$$ W_p(f dx,\overline{f} dx) \leq W_p(f dx,g dx) + W_p(g dx,\overline{f}dx).$$
The symbol $e^{t\Delta}f$ is used to denote the solution of the heat equation after $t$ units of time (since we work on $\mathbb{T}^d$, we
do not need to specify boundary conditions). We state the result of this section as a Lemma.

\begin{lem} Given a nonnegative function $f:\mathbb{T}^d \rightarrow \mathbb{R}_{\geq 0}$ with mean value $\overline{f}$, we have, for all $1 \leq p < \infty$ and a constant $c_{d,p} > 0$
$$ W_p(f dx,\overline{f}dx)  \lesssim_p \frac{\|f\|_{L^1}^{\frac{1}{p}}}{n} + W_p\left( \sum_{  \|k\| \leq c_{d,p} n \log{n}}{\widehat{f}(k) e^{-\|k\|^2/n^2} e^{2\pi i \left\langle k, x \right\rangle}}, \overline{f}dx\right).$$
\end{lem}
We see that we can make the first term as small as we wish by making $n$ sufficiently large; in particular, understanding the behavior of Fourier series truncated at a certain degree is sufficient (as long as the estimates do not depend on the degree of the trigonometric polynomial).

\begin{proof}
We recall that the solution of the heat equation can be written in two different ways: the first way is on the Fourier side
$$ e^{t\Delta}f(x) = \sum_{k \in \mathbb{Z}^d}{\widehat{f}(k) e^{-\|k\|^2 t} e^{2\pi i \left\langle k, x \right\rangle}}$$
while, essentially by linearity of the heat equation,
$$ e^{t\Delta}f(x)  = \left( e^{t\Delta} \delta_0\right) * f,$$
where $e^{t\Delta}\delta_0$ is the heat kernel. It looks, for small $t$, essentially like a Gaussian at scale $\sim \sqrt{t}$ (and converges in profile as $t \rightarrow 0^+$). In particular, the heat kernel on the spatial side behaves much better than would be required for our purposes here. We set $g = e^{t\Delta}f$ and start by estimating $W_p( f,g)$. This is fairly easy by using homogeneity of space and the definition of the heat flow via convolution
$$ W_p(fdx,gdx)^p \leq  \|f\|_{L^1(\mathbb{T}^d)} \int_{\mathbb{T}^d}{ \|x\|^p  \left[ e^{t\Delta} \delta_0\right](x) dx} \lesssim_{p,d} \| f\|_{L^1(\mathbb{T}^d)} t^{\frac{p}{2}}$$
The next step consists of estimating $W_p(g dx, \overline{f} dx)$.
 We set $g = g_1 + g_2$, where
$$ g_1 =  \sum_{k \in \mathbb{Z}^d \atop  \|k\| \leq X}{\widehat{f}(k) e^{-\|k\|^2 t} e^{2\pi i \left\langle k, x \right\rangle}} \quad \mbox{and} 
\quad g_2 =  \sum_{k \in \mathbb{Z}^d \atop  \|k\| > X}{\widehat{f}(k) e^{-\|k\|^2 t} e^{2\pi i \left\langle k, x \right\rangle}},$$
where $X$ is a quantity to be determined further below. We recall the trivial estimate
$$ W_p^p(f(x)dx, \overline{f} dx) \leq \mbox{diam}(M) \|f - \overline{f}\|_{L^1}$$
to argue that
\begin{align*}
W_p(g dx,\overline{f} dx) &\leq W_p(g dx, g_1 dx) + W_p(g_1 dx, \overline{f} dx)\\
&\lesssim_d W_p(g_1 dx, \overline{f} dx) + \|g_2\|_{L^1}.
\end{align*}
We therefore need to estimate the $L^1-$norm of the high-frequency function
\begin{align*}
\left\| \sum_{k \in \mathbb{Z}^d \atop \|k\| \geq X}{\widehat{f}(k) e^{-\|k\|^2 t} e^{2\pi i \left\langle k, x \right\rangle}} \right\|_{L^1(\mathbb{T}^d)} &\leq \left(\sup_{k \in \mathbb{Z}^d}{|\widehat{f}(k)|}\right) \sum_{\|k\| \geq X}{  e^{-\|k\|^2 t} } \\
&\lesssim_d \|f\|_{L^1(\mathbb{T}^d)}  \int_{X}^{\infty}{ e^{-t r^2} r^{d-1} dr}.
\end{align*}
This is a second kind of error that is being introduced: we want to choose $X$ as small as possible while still dominating the first kind of error, $\|f\|_{L^1} t^{p/2}$,  that we already introduced.
We will now show that for
$$  X = c_{d,p} \frac{\log{(1/t)}}{\sqrt{t}} \qquad \mbox{we have} \qquad \int_{X}^{\infty}{ e^{-t r^2} r^{d-1} dr} \lesssim t^{\frac{p}{2}}.$$
A simple substitution shows that
$$  \int_{X}^{\infty}{ e^{-t r^2} r^{d-1} dr} = t^{-\frac{d}{2}} \int_{X\sqrt{t}}^{\infty}{e^{-x^2} x^{d-1} dx}.$$
This integral is also known as the incomplete gamma function $\Gamma(d, X\sqrt{t})$ and we are interested in the regime $X\sqrt{t} \gg 1$. In this regime, there is a classical asymptotic (see e.g. Abramowitz \& Stegun \cite[\S 6.5.35]{abra}), valid for $a \gg 1$, stating that $\Gamma\left(d, a \right) \lesssim_{d} a^{d-\frac{1}{2}} e^{-a}$. Introducing a new variable $n = t^{-1/2}$ leads to the desired result. \end{proof}
A very similar argument can be carried out on general manifolds with eigenfunctions of the Laplacians replacing the trigonometric system; the argument is more or less analogous, some of the new ingredients that would be necessary (estimates on the decay of Green function) are used in \S 9.

\section{Proof of Theorem 1}
\begin{proof} Using the general smoothing procedure outlined in Section 4.2., it suffices to deal with perturbations of the form $d\mu = f(x)dx$ where
$$ f(x) = \sum_{k=1}^{N}{a_{k} \sin{(kx)} + b_{k} \cos{(kx)} }$$
as long as the argument does not depend on $N$. We note that we can assume that $f$ has mean value 0 since $f dx$ is a perturbation of probability measure.
We will also assume that $b_{k} = 0$ and deal with the sine series. The cosine series can be dealt with in the same way, the cases are combined using the triangle inequality. The bound is
derived from an explicit but nontrivial transport plan. This transport compiles a frequency-by-frequency statistics and then looks at the aggregate effect. It can be described as follows (see Fig. 4): we fix $x$ and determine mass transfer to $y$ as follows.

\begin{enumerate} 
\item We go through all the frequencies $a_k \sin{kx}$. If $a_k \sin{kx} > 0$, then we divide that positive mass into two equally
sized positive masses $(a_k/2)\sin{kx}$ and spread them \textit{non-uniformly} over the two adjacent valleys (see Fig. 4). 
\item More precisely, if $y$ is a point in an adjacent valley, then the amount of mass transported from $x$ to $y$ is set to be $ (a_k/4) \sin{(kx)} (-\sin{k y}) dy$.
This maps every frequency to 0.
\item Finally, we look at all possible pairs of points $(x,y)$ and study the total amount of mass $m_1$ that is being transported from $x$ to $y$ and the total amount of mass $m_2$ from $y$ to
$x$. Clearly, if both $m_1$ and $m_2$ are positive, then there is a bit of redundancy happening; we remedy it to by only transporting $|m_1 - m_2|$ wherever it needs to go (the direction is not important: the transportation cost is $|m_1 - m_2||x-y|^p$).
\end{enumerate}

The way the transport plan is set up, there is one favorable cancellation early on in the computation of the cost that simplifies matters. 
However, it is quite conceivable that an entirely different
transport plan could yield better results. 
We introduce, for the purpose of abbreviation, the set
$$c(x,y) = \left\{ k \in \mathbb{N}: \sin{(kx)}~\mbox{has exactly one root between}~x~\mbox{and}~y\right\}.$$
We compute the amount of mass moved from $x$ to $y$ and see that it is given by
$$ \sum_{ k \in c(x,y)}{ \frac{ a_k}{4}  k\sin{(kx)} (-\sin{(ky)}) 1_{a_k \sin{(kx)} > 0}} = \frac{1}{4} \sum_{ k \in c(x,y)}{ k |a_k \sin{(kx)} \sin{(ky)}| 1_{a_k \sin{(kx)} > 0}},$$
since $k \in c(x,y)$ implies that $\sin{(kx)}$ and $\sin{(ky)}$ have opposite sign. Likewise, the amount of mass
transported from $y$ to $x$ is given by
\begin{align*}
\frac{1}{4} \sum_{ k \in c(x,y)}{ ka_k \sin{(ky)} (-\sin{(kx)}) 1_{a_k \sin{(ky)} > 0}} &= \frac{1}{4}\sum_{ k \in c(x,y)}{k |a_k \sin{(kx)} \sin{(ky)}|  1_{a_k \sin{(ky)} > 0} } \\
&= \frac{1}{4}\sum_{k \in c(x,y)}{k |a_k \sin{(kx)} \sin{(ky)}|  1_{a_k \sin{(kx)} < 0} }.
\end{align*}
These two quantities, the total planned transport from $x$ to $y$ and the transport from $y$ to $x$, are nonnegative: however, we clearly do not actually have to transport their sum
but merely their difference (as described in step (3) above). The way the transport plan is set up, one of the absolute values of a term is multiplied with its signature,
and $|x| \sign(x) = x$ has the effect that one of the absolute value signs disappears, which introduces an oscillatory term.
The actual amount of mass $A$ being transported (from either $x$ to $y$ or from $y$ to $x$) is
\begin{align*} 
A &= \left|  \sum_{ k \in c(x,y)}{k |a_k \sin{(kx)} \sin{(ky)}| 1_{a_k \sin{(kx)} > 0}} -  \sum_{k \in c(x,y)}{ k|a_k \sin{(kx)} \sin{(ky)}|  1_{a_k \sin{(kx)} < 0} }\right| \\
&= \left|  \sum_{ k \in c(x,y)}{k |a_k| |\sin{(kx)}| |\sin{(ky)}| (1_{a_k \sin{(kx)} > 0} - 1_{a_k \sin{(kx)} < 0}) }\right| \\
&=  \left|  \sum_{ k \in c(x,y)}{k |a_k| |\sin{(kx)}| |\sin{(ky)}|  \sign\left(a_k \sin{(kx)}\right)} \right| =   \left|  \sum_{ k \in c(x,y)}{ a_k k \sin{(kx)} |\sin{(ky)}| } \right|.
\end{align*}

We use Cauchy-Schwarz to bound the $2p-$th power of the $p-$Wasserstein distance from above
\begin{align*}
 W_p^{p}(\mu ,\mu + f(x)dx)^2 &\leq \left(\int_{\mathbb{T}}{ \int_{\mathbb{T}}{  |x-y|^p \left|  \sum_{ k \in c(x,y)}{ a_k k\sin{(kx)} |\sin{(ky)}| } \right| dx} dy}\right)^2 \\
&\leq 4 \pi^2 \int_{\mathbb{T}}{ \int_{\mathbb{T}}{  |x-y|^{2p}  \left(  \sum_{ k \in c(x,y)}{ a_k k \sin{(kx)} |\sin{(ky)}| } \right)^2 dx} dy}
\end{align*}
We expand the double sum and obtain the upper bound $W_p^p(\mu,\mu + f dx)^2 \leq 4\pi^2 J$ where $J$ can be rewritten as
\begin{align*}
J= \sum_{k, \ell=1}^{N}{ a_k k a_{\ell} \ell \int_{\mathbb{T}}{ \sin{(k x)} \sin{(\ell x)} \int_{\mathbb{T}}{|x-y|^{2p}  |\sin{(ky)}| |\sin{(\ell y)}| 1_{k \in c(x,y)} 1_{\ell \in c(x,y)} dy dx}}}.
\end{align*}

\begin{center}
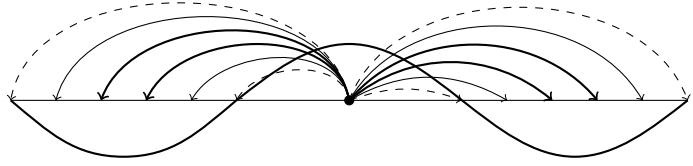
\begin{figure}[h!]
\begin{tikzpicture}[scale=1.5] \label{fig:transport}
\draw[thick] (5,0) sin (6,0.5) cos (7,0) sin (8,-0.5) cos (9,0);
\draw [thick] (3,0) to[out=320,in=180] (4,-0.5) to[out=0,in=220] (5,0);
\draw (3,0) -- (9,0);
\filldraw (6,0) circle (0.04cm);
\draw [dashed, ->] (6,0) to[out=20, in=160] (7,0);
\draw [->] (6,0) to[out=30, in=150] (7.4,0);
\draw [thick, ->] (6,0) to[out=40, in=140] (7.8,0);
\draw [thick, ->] (6,0) to[out=50, in=130] (8.2,0);
\draw [, ->] (6,0) to[out=60, in=120] (8.6,0);
\draw [dashed, ->] (6,0) to[out=70, in=110] (9,0);
\draw [dashed, ->] (6,0) to[out=100, in=60] (5,0);
\draw [->] (6,0) to[out=100, in=60] (4.6,0);
\draw [thick, ->] (6,0) to[out=100, in=65] (4.2,0);
\draw [thick, ->] (6,0) to[out=100, in=70] (3.8,0);
\draw [, ->] (6,0) to[out=100, in=75] (3.4,0);
\draw [dashed, ->] (6,0) to[out=100, in=80] (3,0);
\end{tikzpicture}
\caption{The transport instructions: every positive point-mass of a frequency is spread over the entire adjacent valleys (proportional
to the depth of the valley).} 
\end{figure}
\end{center}

We start by analyzing, for fixed $x \in \mathbb{T}$ and $k, \ell \in \mathbb{N}$, the inner integral
$$ J_2 = \int_{\mathbb{T}}{|x-y|^{2p}  |\sin{(ky)}| |\sin{(\ell y)}| 1_{k \in c(x,y)} 1_{\ell \in c(x,y)} dy }.$$
We assume from now on w.l.o.g. that $k \leq \ell$. For fixed $k,x$, the set $\left\{y: k \in c(x,y) \right\}$ is the union of two intervals of length $\pi/k$ that are distance $\pi/k$ apart. Likewise, the set $\left\{y: \ell \in c(x,y) \right\}$ is the union of two intervals of length $\pi/\ell$ that are distance $\pi/\ell$ apart. 
This implies that
$$ 1_{k \in c(x,y)} 1_{\ell \in c(x,y)} \equiv 0 \qquad \mbox{unless the distance of}~x,y~\mbox{to the nearest root of}~\sin{(k \cdot )}~\mbox{is}~\lesssim \frac{1}{\ell}.$$
This then implies
\begin{align*}
J_2 &\lesssim_p \frac{1}{\ell^{2p}}  \int_{\mathbb{T}}{ |\sin{(ky)}| 1_{k \in c(x,y)} 1_{\ell \in c(x,y)} dy }\lesssim \frac{1}{\ell^{2p}}  \frac{1}{\ell}  \left|\sin{\left(\frac{k}{\ell}\right)}\right| \lesssim \frac{k}{\ell^{2p + 2}}.
\end{align*}
We use this, again together with the fact that $x$ has to be within distance $\sim \ell^{-1}$ of a root of $\sin{(kx)}$ for there to be a nontrivial contribution, to estimate
\begin{align*}
 J &\lesssim \sum_{k, \ell=1\atop k \leq \ell}^{N}{k \ell |a_k|  |a_{\ell}|     \frac{k}{\ell^{2p + 2}}    \int_{\mathbb{T}}{ |\sin{(k x)}| 1_{d(x,{\tiny \mbox{root of}}~\sin{(kx)} \lesssim \ell^{-1})} dx}}\\
&\leq \sum_{k, \ell=1 \atop k \leq \ell}^{N}{ k \ell |a_k|  |a_{\ell}|   \frac{k}{\ell^{2p + 2}}   \frac{k}{\ell}  \left|\sin{\left(\frac{k}{\ell}\right)}\right| } \lesssim  \sum_{k, \ell=1}^{N}{ |a_k| |a_{\ell}| \frac{\min(k,l)^4}{\max(k,\ell)^{2p + 3}}  }.
\end{align*}
This can be estimated by
 $$ \sum_{k, \ell=1}^{N}{ |a_k| |a_{\ell}| \frac{\min(k,l)^4}{\max(k,\ell)^{2p + 3}}  } \lesssim  \sum_{k, \ell=1}^{N}{   \frac{1}{k+\ell}  \frac{ |a_k|}{k^{p-1}} \frac{|a_{\ell}|}{\ell^{p-1} }}$$
to which we apply Hilbert's inequality for sequences of real numbers $(c_k)_{k=1}^{\infty}$, $(d_k)_{k=1}^{\infty}$
$$ \sum_{k, \ell =1}^{\infty}{ \frac{ c_k d_{\ell}}{k + \ell}} \leq \pi \left( \sum_{k=1}^{\infty}{ c_k^2} \right)^{\frac{1}{2}} \left( \sum_{\ell=1}^{\infty}{ d_{\ell}^2} \right)^{\frac{1}{2}}$$
to obtain
$$ W_p^p(\mu,\mu + f(x)dx)^2  \lesssim  \sum_{k, \ell=1}^{N}{   \frac{1}{k+\ell}  \frac{ |a_k|}{k^{p-1}} \frac{|a_{\ell}|}{\ell^{p-1} }} \lesssim  \sum_{k=1}^{\infty}{ \frac{a_k^2}{k^{2p-2}}}.$$ 
The argument is identical for a pure cosine series. This shows that
$$  W_p^p(\mu,\mu + f(x)dx)^{2p} \lesssim_{p}   \sum_{k=1}^{\infty}{ \frac{a_k^2 + b_k^2}{k^{2p-2}}}$$
and therefore
$$  W_p(\mu,\mu + f(x)dx) \lesssim_{p} \left( \sum_{k=1}^{\infty}{ \frac{|\widehat{f}(k)|^2}{k^{2p-2}}} \right)^{\frac{1}{2p}}.$$
The full statement follows from combining this result with the argument from \S 4.2.
\end{proof}

\section{Proof of Theorem 2}
The main idea of the proof is as follows: given a mass distribution $f(x)dx$, it seems rather difficult to get any good estimates on
the Wasserstein distance from below. However, if we can identify a subregion of space where there is more mass than $\overline{f} dx$,
then we have some information: all the 'superfluous' mass has to be transported to a region outside. One would assume that this is
more or less the only restriction (except for 'not too mass is allowed to pile outside the region') and making this intuitive notion precise
(or showing it to be false) seems like a fascinating question that is instrumental to the development of stable lower bounds on Wasserstein distances.\\

Our argument makes use of the following elementary Lemma (versions and presumably stronger versions of which have surely been stated somewhere in the literature): it states that
in our setting we can approximate any transport plan by a piecewise linear function.
\begin{lem} Let $f: \mathbb{T} \rightarrow \mathbb{R}_{\geq 0}$ be given and have mean value $\overline{f}$. For any $\varepsilon > 0$, there exists a piecewise linear function $T:\mathbb{T} \rightarrow \mathbb{T}$ such that
$$ \int_{\mathbb{T}}{ |T(x) - x| f(x) dx} \leq W_1(f(x)dx, \overline{f} dx) + \varepsilon$$
and where, denoting the push-forward of the measure $f(x)dx$ under the mapping $T$ by $\nu$, we have $W_1(\nu, \overline{f} dx) \leq \varepsilon$.
\end{lem}
\begin{proof} Let $\gamma \in \Gamma(f(x)dx, \overline{f} dx)$ be a transport plan that is at most $\varepsilon/2$ more costly than the optimal transport cost $W_1(f(x)dx, \overline{f} dx)$. 
We subdivide $\mathbb{T}$ into $n$ intervals of length $2\pi/n$ and can then 'integrate out the marginals' or, more directly, determine how much mass the transport plan $\gamma$ moves from each of the intervals to each of the other intervals. Then, however, the construction is fairly straightforward. Suppose the interval $I_1$ has total mass $m_1$ and we would like to move mass $m_{1,j}$ from $I_1$ to interval $I_j$, where $1 \leq j \leq n$.
Clearly,
$$ m_1 = m_{1,1} + m_{1,2} + \dots, m_{1,n},$$
where
$$ m_{i,j} = \int_{I_i} \int_{I_j} d\gamma(x,y).$$
We now partition the interval $I_1$ into $n$ Intervals $I_{1,k}$ for $1 \leq k \leq n$ such that
$$ \int_{I_{1,k}}{f(x)dx} = m_{1,k}.$$
We note that this is always possible because $f(x)dx$ does is absolutely continuous. The piecewise linear map $T$ can then be defined by requiring that $T(I_{1,k}) = I_{k}$. The resulting pushforward measure has the correct amount of mass in every single interval $I_j$, the worst case is that we have to redistribute within the interval which has total $W_1-$cost at most $2 \pi \overline{f} 2 \pi n^{-1}$. Since $n$ was arbitrary, this can be made arbitrarily small.
\end{proof}

The proof gives slightly more information, it can be used to show that $T$ can actually be chosen to be piecewise linear on a finite number of intervals whose number depends only on $\varepsilon$. At first glance, this may look like a simplified form of the transport problem, however, problems in optimal transport are usually fairly easy to \textit{approximately} solve after discretizing space and turning it into a finite combinatorial matching problem. We do not believe the simplified map $T$ to yield much insight in general (conversely, our proof could have also bypassed $T$ by incorporating the construction in the proof above; we have found this way simpler).

\begin{proof}[Proof of Theorem 2]
 We assume that $f: \mathbb{T} \rightarrow \mathbb{R}_{\geq 0}$ is given and has mean value $\overline{f}$. Our goal is to obtain a lower bound on $W_1(f(x)dx, \overline{f} dx)$. Our general smoothing procedure allows us to assume w.l.o.g. that $f$ is given as 
$$ f(x) = \sum_{|k| \leq N}{\widehat{f}(k) e^{i k x}}$$
as long as our subsequent estimates do not depend on $N$: we can always take $N$ so large that $\|f\|_{L^1}/N \ll W_1(f(x)dx, \overline{f} dx)$ and then apply the smoothing procedure.
We will, mainly for simplicity of notation, construct a lower bound on the simple class of transport plans that is given by explicit functions $T:\mathbb{T} \rightarrow \mathbb{T}$ as opposed to the actual
set of measures on the product space -- this is justified by the Lemma above.
Assuming $T:\mathbb{T} \rightarrow \mathbb{T}$ is one of these specialized transport plans, the cost of transport to equilibrium is computed as
$$ \mbox{transport cost} =  \int_{\mathbb{T}}{ |T(x) - x| f(x) dx}.$$
Let now $t > 0$ be arbitrary; we can introduce an additional layer of averaging by writing
$$  \int_{\mathbb{T}}{ |T(x) - x| f(x) dx} =  \int_{\mathbb{T}}{  \frac{1}{t} \int_{x}^{x+t}{|T(y) - y| f(y) dy} dx}.$$
The next step is to obtain a lower bound on the inner integral assuming that, for $x$ fixed,
$$  \int_{x}^{x+t}{f(y) dy} \geq t \overline{f},$$
which merely states that there is too much mass inside the interval $[x, x+t]$. Clearly, it has to be transported out, and we now obtain a uniform bound on how cheaply that can be done, that is independent of the transport map $T$. Clearly, the best possible case is if all the superfluous mass is already as close to the boundary as possible: in that case, invoking the bound on the maximal density $\|f\|_{L^{\infty}}$ we obtain a lower bound for the local transportation cost
$$  \int_{x}^{x+t}{|T(y) - y| f(y) dy} \gtrsim \frac{1}{\|f\|_{L^{\infty}(\mathbb{T})}}  \left( \int_{x}^{x+t}{f(y) dy} - t \overline{f} \right)^2$$
as follows: the amount of 'superfluous' mass in the interval $[x,x+t]$ is given by
$$ m_{+} =  \int_{x}^{x+t}{f(y) dy} - t \overline{f}.$$
To minimize the cost of transporting it outside, we would assume it to be arranged as close to the boundary as possible. Meanwhile, there is
also a bound on the $L^{\infty}$ norm which forces it to be arranged on an interval of length at least $\sim m_{+}/\|f\|_{L^{\infty}}$. Transporting it outside then means that at least half of that mass has to be transported at least half of the length of that interval leading to
a lower bound of $\gtrsim m_{+}^2/\|f\|_{L^{\infty}}$ as was claimed.
\begin{center}
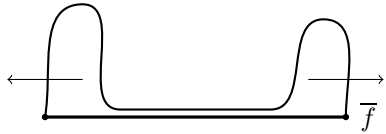
\begin{figure}[h!]
\begin{tikzpicture}
\draw [very thick] (5,0) -- (9,0);
\draw [thick] (5,0) to[out=80, in=180] (5.5,1.5) to[out=0, in=180] (6,0.1) to[out=0, in=180] (8,0.1) to[out=0, in=180] (8.7,1.3) to[out=0, in=90] (9,0);
\draw [->] (5.5, 0.5) -- (4.5, 0.5);
\draw [->] (8.5, 0.5) -- (9.5, 0.5);
\node at (9.3, 0) {$\overline{f}$};
\filldraw (5,0) circle (0.035cm);
\filldraw (9,0) circle (0.035cm);
\end{tikzpicture}
\caption{The interval $J$: bounding mass concentration near the boundary.}
\end{figure}
\end{center}
We recall that this inequality is conditional on the amount of mass in $[x,x+t]$ exceeding the average $\overline{f}$. By double-counting the transport plan
and looking at 'import' instead of 'export', we get the same inequality if the mass is smaller than $t \overline{f}$. Therefore,
\begin{align*}
  \int_{\mathbb{T}}{ |T(x) - x| f(x) dx} &=  \int_{\mathbb{T}}{  \frac{1}{t} \int_{x}^{x+t}{|T(y) - y| f(y) dy} dx} \\
&\gtrsim  \frac{1}{\|f\|_{L^{\infty}(\mathbb{T})}}  \int_{\mathbb{T}}{  \frac{1}{t}  \left( \int_{x}^{x+t}{f(y) dy} - t \overline{f} \right)^2 dt}.
\end{align*}
An expansion in Fourier series shows that
$$  \int_{x}^{x+t}{f(y) dy} - t \overline{f} = \sum_{|k| \leq N \atop k \neq 0}{\widehat{f}(k) \frac{ e^{i k x}}{ik} \left( e^{ikt} -1 \right)  }$$  
and thus
$$    \frac{1}{ \|f\|_{L^{\infty}(\mathbb{T})}}  \int_{\mathbb{T}}{\frac{1}{t} \left( \int_{x}^{x+t}{f(y) dy} - t \overline{f} \right)^2 dx} =   \frac{1}{ \|f\|_{L^{\infty}(\mathbb{T})}}  \frac{1}{t} \sum_{|k| \leq N \atop k \neq 0}{\frac{|\widehat{f}(k)|^2 }{k^2} \left| e^{ikt} -1 \right|^2  }.$$
This is valid for all $t > 0$. We can therefore average over $0 \leq t \leq \pi$ and obtain
\begin{align*}
\frac{1}{\pi}\int_{0}^{\pi}{ \frac{1}{t \|f\|_{L^{\infty}(\mathbb{T})}}  \sum_{|k| \leq N \atop k \neq 0}{\frac{|\widehat{f}(k)|^2 }{k^2} \left| e^{ikt} -1 \right|^2  } dt}      \gtrsim \frac{1}{\|f\|_{L^{\infty}(\mathbb{T})}}  \sum_{|k| \leq N \atop k \neq 0}{\frac{|\widehat{f}(k)|^2 }{k^2} \int_{0}^{\pi}{ \frac{\left| e^{ikt} -1 \right|^2}{t} dt}  }.
\end{align*}
This integral is easy to bound: the integrand can be simplified as 
\begin{align*}
\int_{0}^{\pi}{ \frac{\left| e^{ikt} -1 \right|^2}{t} dt} = \int_{0}^{\pi}{ \frac{2 - 2\cos{(kt)}}{t} dt} &\sim \sum_{\ell = 1}^{k}{ \int_{\frac{\pi \ell}{k}}^{\frac{\pi (\ell+1)}{k}}{ \frac{2 - 2\cos{(kt)}}{t} dt}} \\
&\sim  \sum_{\ell = 1}^{k}{ \frac{1}{k} \frac{k}{\pi (\ell+1)}} \sim \log{k}.
\end{align*}
This implies the desired result.
\end{proof}

\section{Proof of Theorem 3}
\begin{proof} We only discuss the proof for $\alpha = \sqrt{2}$, the general proof is completely identical, since we only use the fact that $\sqrt{2}$ is badly approximable (but, of course, all implicit
constants will depend on $\alpha$).
The Fourier coefficients are easily bounded since the exponential sum is a geometric series
$$ \left| \widehat{\mu_N}(k) \right| = \left|  \frac{1}{N}  \sum_{n=1}^{N}{e^{2 \pi i k n\sqrt{2}}} \right| = \frac{1}{N} \left| \frac{e^{2\pi k (N+1) \sqrt{2} } - 1}{e^{2\pi i k \sqrt{2}} - 1} \right| \leq \frac{1}{N} \frac{2}{\left| e^{2\pi i k \sqrt{2}} - 1 \right|} \lesssim \frac{1}{N} \frac{1}{ \left\| k \sqrt{2} \right\|},$$
where $\| x \| = \min(x-\left\lfloor x \right\rfloor, \left\lceil x \right\rceil - x)$ is the distance to the nearest integer. The problem thus essentially reduces to
studying the distribution of $\left\| k \sqrt{2} \right\|$. Here, we can use a classic algebraic trick to conclude that for any integer $a \in \mathbb{N}$
$$ \left| (k\sqrt{2} - a)(k \sqrt{2} + a) \right| = \left| 2 k^2 - a^2  \right| \geq 1 \qquad \mbox{and thus} \qquad \|k\sqrt{2}\| \gtrsim \frac{1}{|k|}.$$
This step does not work for general $\alpha$ but the inequality is still true and is equivalent to the definition of a badly approximable number.
We observe the following implication: the set
$$ \left\{ \|k \sqrt{2}\| : 2^{\ell} \leq k \leq 2^{\ell + 1}\right\}$$
is $\sim 2^{-\ell}$ separated and every element is bounded away from 0 and satisfies $\gtrsim 2^{-\ell}$. This allows us to bound
\begin{align*}
\frac{1}{n} + \sum_{k=1}^{n}{ \frac{ \left| \widehat{\mu_N}(k) \right|}{k} } &\lesssim \frac{1}{n} + \sum_{\ell=1}^{\log{n}}{ \sum_{2^{\ell} \leq k \leq 2^{\ell+1}}{  \frac{ \left| \widehat{\mu_N}(k) \right|}{k}}} \lesssim \frac{1}{n} + \sum_{\ell=1}^{\log{n}}{  \frac{1}{2^{\ell}}  \sum_{2^{\ell} \leq k \leq 2^{\ell+1}}{ \left| \widehat{\mu_N}(k) \right|}}\\
 &\lesssim \frac{1}{n} + \sum_{\ell=1}^{\log{n}}{  \frac{1}{2^{\ell}} \frac{1}{N} \sum_{2^{\ell} \leq k \leq 2^{\ell+1}}{ \frac{1}{\|k\sqrt{2}\|}}} \lesssim \frac{1}{n} + \sum_{\ell=1}^{\log{n}}{  \frac{1}{2^{\ell}} \frac{1}{N} \sum_{1 \leq m \leq 2^{\ell}}{ \frac{1}{(m/2^{\ell})}}} \\
 &\lesssim \frac{1}{n} + \frac{1}{N} \sum_{\ell=1}^{\log{n}}{ \ell} \sim \frac{1}{n} + \frac{(\log{n})^2}{N}
\end{align*}
and optimizing in $n$ leads to the choice $n=N$. The same kind of reasoning yields slightly better bounds for the Wasserstein distance. We use $|\widehat{\mu(k)}| \leq \|\mu\|_{L^1} = 1$ to conclude
\begin{align*}
 W_2(\mu_N, dx) \lesssim  \left(\sum_{k=1}^{\infty}{ \frac{ \left| \widehat{\mu_N}(k) \right|^2}{k^2} } \right)^{\frac{1}{2}} \leq   
\left(\sum_{k=1}^{N^2}{ \frac{ \left| \widehat{\mu_N}(k) \right|^2}{k^2} }   + \sum_{k=N^2+1}^{\infty}{ \frac{1}{k^2} }      \right)^{\frac{1}{2}} \leq \frac{1}{N} +  \left(\sum_{k=1}^{N^2}{ \frac{ \left| \widehat{\mu_N}(k) \right|^2}{k^2}} \right)^{\frac{1}{2}}
\end{align*}
We estimate, as above,
\begin{align*}
\sum_{k=1}^{N^2}{ \frac{ \left| \widehat{\mu_N}(k) \right|^2}{k^2} }&\lesssim  \sum_{\ell=1}^{\log{(N^2)}}{ \sum_{2^{\ell} \leq k \leq 2^{\ell+1}}{ \frac{ \left| \widehat{\mu_N}(k) \right|^2}{k^2} }} \lesssim  \sum_{\ell=1}^{\log{(N^2)}}{  \frac{1}{2^{2\ell}} \sum_{2^{\ell} \leq k \leq 2^{\ell+1}}{ \left| \widehat{\mu_N}(k) \right|^2}}\\
 &\lesssim  \frac{1}{N^2}\sum_{\ell=1}^{\log{(N^2)}}{  \frac{1}{2^{2\ell}} \sum_{2^{\ell} \leq k \leq 2^{\ell+1}}{ \frac{1}{\|k\sqrt{2}\|^2}}} \lesssim  \frac{1}{N^2} \sum_{\ell=1}^{\log{(N^2)}}{  \frac{1}{2^{2\ell}} \sum_{1 \leq m \leq 2^{\ell}}{ \frac{1}{(m/2^{\ell})^2}}} \\
 &\lesssim  \frac{1}{N^2}\sum_{\ell=1}^{\log{(N^2)}}{  \frac{1}{2^{2\ell}} \sum_{m=1}^{\infty}{ \frac{1}{(m/2^{\ell})^2}}}  \lesssim   \frac{1}{N^2}\sum_{\ell=1}^{\log{(N^2)}}{ 1} \sim \frac{ \log{(N^2)}}{N^2}
\end{align*}
and therefore
$$ W_2(\mu_N, dx) \lesssim \frac{\sqrt{\log{N}}}{N}.$$
\end{proof}

\section{Proof of Theorem 4}
\begin{proof}[Proof of Theorem 4]
We will first prove that if $f:\mathbb{T} \rightarrow \mathbb{R}_{\geq 0}$ is continuous with mean $\overline{f}$, then
$$ \frac{\|f - \overline{f}\|^2_{L^1(\mathbb{T})}}{\|f\|_{L^{\infty}(\mathbb{T})}} \frac{1}{  \# \left\{ x \in \mathbb{T}: f(x) = \overline{f}\right\} } \lesssim W_1(f(x)dx, \overline{f} dx) \lesssim  \left( \sum_{k=1}^{\infty}{ \frac{ |\widehat{f}(k)|^2}{k^2}}\right)^{\frac{1}{2}}.  $$
Once this is established, we will show how this implies the full desired statement.
The upper bound is fairly straightforward: we use homogeneity, the H\"older inequality and Peyre's estimate to conclude that
\begin{align*}
 W_1(f(x)dx, \overline{f} dx) &= \overline{f} \cdot W_1\left( \frac{f(x)}{\overline{f}} dx, dx\right)  \\
 &\leq 2\overline{f}\cdot W_2\left( \frac{f(x)}{\overline{f}} dx, dx\right)  \leq
2\overline{f} \left\|  \frac{f(x)}{\overline{f}} \right\|_{\dot H^{-1}(\mathbb{T})} =  \left\|  f \right\|_{\dot H^{-1}}.
\end{align*}
The lower bound requires a bit more work. 
We consider intervals $J \subset \mathbb{T}$ that are bounded by two solutions of $f(x) = \overline{f}$ and do not contain any additional solutions of this equation. If $f > \overline{f}$ in $J$, then there is a positive amount of mass $\|f\|_{L^1(J)}$ that needs to be moved out to obtain the constant measure $\overline{f} dx$.
The cost of moving the superfluous the mass inside $J$ depends on its distance to the boundary. We introduce the set
$$ A_J = \left\{x \in J: d(x, J^c) \geq \frac{1}{3}\frac{\|f - \overline{f}\|_{L^1(J)}}{\|f\|_{L^{\infty}(J)}}\right\},~\mbox{note} \quad  \frac{1}{3}\frac{\|f - \overline{f}\|_{L^1(J)}}{\|f\|_{L^{\infty}(J)}} \leq \frac{|J|}{3},$$
as well as
$$ |J \setminus A_J|  = \left| \left\{ x \in J: d(x, J^c) < \frac{1}{3}\frac{\|f - \overline{f}\|_{L^1(J)}}{\|f\|_{L^{\infty}(J)}}\right\} \right|   =  \frac{2}{3}\frac{\|f - \overline{f}\|_{L^1(J)}}{\|f\|_{L^{\infty}(J)}}$$
and claim
$$ \int_{A_J}{ f(x) dx} \geq \frac{1}{3} \|f - \overline{f}\|_{L^1(J)}.$$
This can be seen as follows
\begin{align*}
 \|f - \overline{f}\|_{L^1(J)}  &= \int_{A_J}{ (f(x) - \overline{f})dx} + \int_{J \setminus A_J}{ (f(x) - \overline{f} )dx}\\
&\leq \int_{A_J}{( f(x) - \overline{f}) dx}  + \left| J \setminus A_J \right| \|f\|_{L^{\infty}(J)} \\
&= \int_{A_J}{ (f(x) - \overline{f} )dx}  + \frac{2}{3} \|f - \overline{f}\|_{L^1(J)}.
\end{align*}
As a consequence
$$ W_1(f(x) dx,\overline{f} dx) \geq \sum_{J: f\big|_{J} > \overline{f}}{  \frac{1}{3} \|f-\overline{f}\|_{L^1(J)}    \left(\frac{1}{3}\frac{\|f-\overline{f}\|_{L^1(J)}}{\|f\|_{L^{\infty}(J)}}\right)^{}} =
 \frac{1}{9} \sum_{J: f\big|_{J} > \overline{f}}{  \frac{\|f-\overline{f}\|^{2}_{L^1(J)}}{\|f\|_{L^{\infty}(J)}}},$$
where the sum ranges over all intervals $J$ on which $f-\overline{f}$ is positive. Of course, the same argument can be reversed (this is similar to the proof of Theorem 2 above) and we can
take intervals where the average is smaller than $\overline{f}$ and argue in terms of mass that has to be imported and
$$ W_1(f(x) dx,\overline{f} dx) \gtrsim   \sum_{J: f\big|_{J} < \overline{f}}{  \frac{\|f-\overline{f}\|^{2}_{L^1(J)}}{\overline{f}}} \gtrsim  \sum_{J: f\big|_{J} < \overline{f}}{  \frac{\|f-\overline{f}\|^{2}_{L^1(J)}}{\|f\|_{L^{\infty}(J)}}}.$$
Altogether, this implies
yields
$$ W_1(f(x) dx,\overline{f} dx) \gtrsim \frac{1}{\|f\|_{L^{\infty}(\mathbb{T})}}   \sum_{J}{  \|f-\overline{f}\|^{2}_{L^1(J)}}.$$
The remaining sum is easy to bound via Cauchy-Schwarz
$$ \|f - \overline{f}\|_{L^1(\mathbb{T})} =  \sum_{J}{  \|f-\overline{f}\|^{}_{L^1(J)}} \leq \left(  \sum_{J}{  \|f-\overline{f}\|^{2}_{L^1(J)}} \right)^{\frac{1}{2}} \left(\# J\right)^{\frac{1}{2}}.$$
The number $\#J$ of such intervals obviously equals the number of solutions of $f(x) = \overline{f}$
$$ W_1(f(x) dx,\overline{f} dx) \gtrsim \frac{ \|f - \overline{f}\|^{2}_{L^1(\mathbb{T})} }{{\|f\|_{L^{\infty}(\mathbb{T})}}}  \frac{1}{ \# \left\{ x \in \mathbb{T}: f(x) = \overline{f}\right\}}.$$
By rearranging, this establishes the following inequality for nonnegative, continuous functions $f:\mathbb{T} \rightarrow \mathbb{R}_{\geq 0}$
$$ \frac{\|f - \overline{f}\|^2_{L^1(\mathbb{T})}}{\|f\|_{L^{\infty}(\mathbb{T})}} \lesssim   \# \left\{ x \in \mathbb{T}: f(x) = \overline{f}\right\} \left( \sum_{k=1}^{\infty}{ \frac{ |\widehat{f}(k)|^2}{k^2}}\right)^{\frac{1}{2}}.  $$
Let now $g: \mathbb{T} \rightarrow \mathbb{R}_{\geq 0}$ be continuous with mean value $0$. We abbreviate
$$ m = \left| \min_{x \in \mathbb{T}} g(x) \right| \qquad \mbox{and consider} \qquad f(x) = g(x) + m \geq 0.$$
Clearly, $\overline{f} = m$ and thus we can apply the inequality above
$$  \# \left\{ x \in \mathbb{T}: f(x) = \overline{f}\right\} \left( \sum_{k=1}^{\infty}{ \frac{ |\widehat{f}(k)|^2}{k^2}}\right)^{\frac{1}{2}} = 
 \# \left\{ x \in \mathbb{T}: g(x) =0 \right\} \left( \sum_{k=1}^{\infty}{ \frac{ |\widehat{g}(k)|^2}{k^2}}\right)^{\frac{1}{2}}.$$
Furthermore, 
$$  \|f - \overline{f}\|^2_{L^1(\mathbb{T})} = \|g\|^2_{L^1(\mathbb{T})} \qquad \mbox{and} \qquad  \|f\|_{L^{\infty}(\mathbb{T})}   \leq 2 \|g\|_{L^{\infty}(\mathbb{T})}$$
and therefore
$$   \frac{\|f - \overline{f}\|^2_{L^1(\mathbb{T})}}{\|f\|_{L^{\infty}(\mathbb{T})}}  \geq \frac{1}{2}  \frac{\|g\|^2_{L^1(\mathbb{T})}}{\|g\|_{L^{\infty}(\mathbb{T})}}.$$
\end{proof}

We quickly remark that there is an elementary argument that shows the simpler inequality
$$ (\mbox{number of roots of}~f) \cdot \left( \sum_{k=1}^{\infty}{ \frac{ |\widehat{f}(k)|}{k}}\right)^{}  \gtrsim \|f\|^{}_{L^1(\mathbb{T})}.$$
follows quickly from writing the $L^1-$norm of $f$ as 
$$ \|f\|_{L^1} = \sum_{k}{  \left| \int_{r_k}^{r_{k+1}}{ f(x) dx} \right|},$$
where the sum ranges over the roots $r_k$ ordered in increasing magnitude. We can then use this in combination with the uniform bound
$$  \left| \int_{r_k}^{r_{k+1}}{ f(x) dx} \right| \lesssim  \sum_{k=1}^{\infty}{ \frac{ |\widehat{f}(k)|}{k}}$$
to obtain the desired result. It could be interesting to understand whether there are other inequalities in the style of this one or Theorem 4.

\section{Proof of Theorem 5}
The main idea behind the proof of Theorem 5 is a reinterpretation of parabolic second order differential equations, more specifically the heat equation
$$ \left(\frac{\partial}{\partial t} - \Delta\right) u(t,x) = 0.$$ 
We observe that if $u(0,x) \geq 0$, then this process can actually be reinterpreted in physical terms: we have a nonnegative density of particles that starts moving and spreading out (this is the actual underlying idea behind the propagation of heat). We know from the spectral expansion that high-frequency eigenfunctions and linear combinations thereof have fast decay under the heat equation: within a short amount of time, the solution $u(t,x)$ will be close to constant. At the same time, if we only have to run the equation for a short time, then most of the particles are only moving a little bit. This intuition could be made precise in terms of stochastic formulations of the heat equation (by actually taking expectations over travel time of Brownian motion), we opted for a more elementary presentation in terms of Green's functions. The argument already had another application in \cite{steini}.

\begin{proof} Let $(M,g)$ denote a compact, smooth Riemannian manifold without boundary. We denote the $L^2-$normalized eigenfunctions of the
Laplacian as $(\phi_k)_{k=0}^{\infty}$ (where, as usual, $\phi_0 = (\vol(M))^{-1/2}$ is constant). Let now, for some $\lambda > 0$, $f:M \rightarrow \mathbb{R}$ denote
some continuous function whose spectral expansion is only comprised of eigenfunctions above a certain frequency $\lambda$
$$ f = \sum_{\lambda_k \geq \lambda}{ \left\langle f, \phi_k \right\rangle \phi_k}.$$
The argument is based on an explicit construction comprised of two steps: we first use the heat kernel as a way to organize transport to achieve a re-distribution
of mass that is very close to flat; the second step is the trivial $L^1-$estimate 
$$ W_p^p(f(x)dx, \overline{f} dx) \leq \mbox{diam}(M) \|f - \overline{f}\|_{L^1}.$$
The heat kernel $p(t,x,y): \mathbb{R}_{\geq 0} \times M \times M \rightarrow \mathbb{R}_{\geq 0}$ satisfies 
$$ \int_{M}{ p(t,x,y) dy} = 1.$$
In particular, it may be understood as a probability distribution. We re-interpret it as a transport plan telling us how to spread mass located at $x$. Linearity of integral operators
allows us to think of both positive mass $\max\left\{f, 0\right\}$ and negative mass $\min\left\{f, 0\right\}$ moving separately. 
The result of this
transport plan will be a new mass distribution given by
$$ g(x) = \int_{M}{ p(t,x,y) f(y) dy} \quad \mbox{at $W_p$-cost} \quad \left(\int_{M}{ \int_{M}{ |x-y|^p p(t,x,y) |f(y)| dy} dx}\right)^{\frac{1}{p}}.$$
This transportation cost is easy to bound: we use a classical bound of Aronson \cite{aronson}
$$ p(t,x,y) \leq \frac{c_1}{t^{n/2}} \exp \left( -\frac{|x-y|^2}{c_2 t} \right)$$
for some $c_1, c_2$ depending only on $(M,g)$ and obtain
$$ \int_{M}{ \int_{M}{ |x-y|^p p(t,x,y) |f(y)| dy} dx} \lesssim_{(M,g)}  \int_{M}{ \int_{M}{\frac{ |x-y|^p}{t^{n/2}} \exp \left( -\frac{|x-y|^2}{c_2 t} \right)|f(y)| dy} dx}.$$
However, it is easily seen that for some universal constants depending on the manifold
$$  \int_{M}{\frac{ |x-y|^p}{t^{n/2}} \exp \left( -\frac{|x-y|^2}{c_2 t} \right) dx} \lesssim_{c_2,p,(M,g)} t^{\frac{p}{2}}.$$
Altogether, this implies the transportation cost is bounded by (since $c_2 \lesssim_{(M,g)} 1$)
$$  \left(\int_{M}{ \int_{M}{ |x-y|^p p(t,x,y) |f(y)| dy} dx}\right)^{\frac{1}{p}} \lesssim_{p, (M,g)}   \sqrt{t} \|f\|_{L^1}^{\frac{1}{p}}.$$
The resulting profile $g(x)$ of this transportation map is easily understood since it is the solution of the heat equation and thus diagonalized by the Laplacian eigenfunctions
and thus
$$ g(x) = \int_{M}{ p(t,x,y) f(y) dy}  = \sum_{\lambda_k \geq \lambda}{ e^{-\lambda_k t} \left\langle f, \phi_k\right\rangle \phi_k}$$
and its $L^1-$norm can be bounded by Cauchy-Schwarz
\begin{align*} \|g\|^2_{L^1(M)} &= \left\| \sum_{\lambda_k \geq \lambda}{ e^{-\lambda_k t} \left\langle f, \phi_k\right\rangle \phi_k} \right\|^2_{L^1(M)} \leq \vol(M) \left\| \sum_{\lambda_k \geq \lambda}{ e^{-\lambda_k t} \left\langle f, \phi_k\right\rangle \phi_k} \right\|^2_{L^2(M)} \\
&= \vol(M) \sum_{\lambda_k \geq \lambda}{ e^{-2 \lambda_k t} \left| \left\langle f, \phi_k\right\rangle \right|^2} \lesssim_{(M,g)} e^{-2\lambda t}  \sum_{\lambda_k \geq \lambda}{ \left| \left\langle f, \phi_k\right\rangle \right|^2}
\lesssim_{(M,g)} e^{-2 \lambda t} \|f\|_{L^2(M)}^2
\end{align*}
and thus we can bound the total transport cost by
$$W_p\left(\max\left\{f(x), 0\right\}dx, \max\left\{-f(x), 0\right\} dx\right)^p \lesssim_{p, (M,g)}  t^{\frac{p}{2}} \|f\|_{L^1(M)} +  e^{- \lambda t} \|f\|_{L^2(M)}.$$
Setting 
$$t = \lambda^{-1} \log{\left( \frac{\lambda^{\frac{p}{2}} \|f\|_{L^2(M)}}{\|f\|_{L^1(M)}}\right)} $$
yields
$$ W_p^p(f,0) \lesssim_p 
\frac{ \left[  \log\left(  \lambda \frac{\|f\|_{L^2(M)}}{\|f\|_{L^1(M)}}   \right)\right]^{\frac{p}{2}}}{ \lambda^{\frac{p}{2}}} \|f\|_{L^1}.$$
If we re-run the argument with the additional information that $f= \phi_k$ is an eigenfunction, then we can use
$$ \int_{M}{ p(t,x,y) \phi_k(y) dy} = e^{-\lambda_k t} \phi_k(x)$$
to avoid the Cauchy-Schwarz inequality and obtain the improved estimate
$$W_p\left(\max\left\{\phi_k(x), 0\right\}dx, \max\left\{-\phi_k(x), 0\right\} dx\right)^p  \lesssim_{p, (M,g)} \left( t^{\frac{p}{2}} +  e^{- \lambda_k t}  \right) \|\phi_k\|_{L^1(M)}.$$
Setting $t = \lambda_k^{-1} \log{\lambda_k}$ yields the desired result
$$ W_p\left(\max\left\{\phi_k(x), 0\right\}dx, \max\left\{-\phi_k(x), 0\right\} dx\right)^p \lesssim_{p, (M,g)} \left(\frac{ \log{\lambda_k}}{\lambda_k}\right)^{\frac{p}{2}} \|\phi_k\|_{L^1(M)}.$$
\end{proof}

\textbf{Acknowledgment.} The author is very grateful to an anonymous referee for a very thorough and detailed reading as well
as several suggestions that greatly improved the quality of the manuscript.

\end{document}